\documentclass[a4paper,oneside,10pt,reqno]{amsart} 
\usepackage[margin=0.5in,a4paper]{geometry}

\usepackage{amsmath,amsfonts,amssymb,txfonts}
\usepackage[all]{xy}
\usepackage{hyperref,color,enumerate,graphicx}

\newtheorem{theorem}{Theorem}[section]
\newtheorem{lemma}[theorem]{Lemma}

\newtheorem{proposition}[theorem]{Proposition}
        \theoremstyle{definition}
          
           \newtheorem{definition}[theorem]{Definition}

\newcommand{\lemref}[1]{Lemma~\ref{#1}}
\newcommand{\propref}[1]{Proposition~\ref{#1}}

\newcommand{\thmref}[1]{Theorem~\ref{#1}}
\newcommand{\secref}[1]{Section~\ref{#1}}

   \newcommand\negro{\normalcolor}
\newcounter{zatia}

\usepackage{stmaryrd}

\newcommand{\TO}{\longrightarrow}

\newcommand{\wt}{\widetilde}

\newcommand{\g}{\noindent {\mathfrak{g}}}
\newcommand{\tq}{\ \big| \ }

\newcommand{\class}{\mathop{\rm class \, } \nolimits}

\newcommand{\Aut}{\mathop{\rm Aut \, } \nolimits}

\newcommand{\Ide}{\mathop{\rm Id \, } \nolimits}

\newcommand{\Ad}{\mathop{\rm Ad \, } \nolimits}

\newcommand\chii{\raise2pt\hbox{$\chi$}}
\newcommand{\rondp}{\raise2pt\hbox{\tiny $\circ$}}

\newcommand\phii{{\raise2pt\hbox{$\varphi$}}}
\newcommand\phibas{{\raise2pt\hbox{\footnotesize $\varphi$}}}
\newcommand\Om{\Omega}
\newcommand\BOm{\underline\Omega}
\newcommand\BE{\underline E}

\newcommand\BZ{\underline Z}

\newcommand\om{\omega}

\newcommand\N{\mathbb N}
\newcommand\R{\mathbb R}

\newcommand\B{\underline B}

\newcommand{\id}{{\large 1}}
\newcommand{\bi}[2]{{#1}^{^{#2}}}
\newcommand{\ib}[2]{{#1}_{_{#2}}}
\newcommand{\hiru}[3]{{#1}^{^{#2}}_{_{#3}}}
\newcommand{\lau}[4]{{#1}^{^{#2}}_{_{#3}}{\left( #4 \right)}}
\newcommand{\coho}[3]{{#1}^{^{#2}}{\left( #3 \right)}}

\title
{Spectral sequence of an isometric action}
\thanks{This work
has been partially supported by Ministerio de Ciencia e Innovaci\'on, Spain, grant PID2022-139631NB-I00. 
The authors acknowledge that the research cooperation was funded by the program Excellence Initiative Research University at the Jagiellonian University in Krakow within the framework of the research group Reeb-Reinhardt 2022.
}

\date{\today}

\author[J.I.~Royo Prieto]{Jos\'{e} Ignacio Royo Prieto}
\address{Matematika Saila\\ Zientzia eta Teknologia Fakultatea\\ University of the Basque Country UPV/EHU\\ Barrio Sarriena s/n\\ 48940 Leioa\\Spain.}
\email{joseignacio.royo@ehu.eus}

\author{Martintxo Saralegi-Aranguren}
\address{Laboratoire de Math{\'e}matiques de Lens\\  
      EA 2462 \\
      Universit\'e d'Artois\\
         SP18, rue Jean Souvraz\\
          62307 Lens Cedex\\
         France}
\email{martin.saraleguiaranguren@univ-artois.fr}

\keywords{Isometric action. Spectral sequence.}
\subjclass[2010]{Primary 57S15; Secondary 55N33.}

\begin{document}

\thispagestyle{empty}

\begin{abstract}
We consider a free smooth action $\Phi \colon G \times M \to M$ of a connected compact Lie group $G$ on a manifold $M$.
We examine the Cartan filtration of the complex of differential forms of $M$. The associated spectral sequence $\hiru{E}{p,q}{r}$ converges to the cohomology of $M$. It is well known that the second page $\hiru{E}{p,q}{2}$  of this spectral sequence is given by $\coho H p {M/G} \otimes \coho H q \g$, where $\g$ denotes the Lie algebra of $G$.

In this note, we provide a straightforward proof of this fact without using Mayer-Vietoris, harmonic operators, or other such methods found in existing proofs. In fact, we extend this result to the case where the action is locally free and  $G$ is not compact, under the hypothesis that $\Phi$ extends to a smooth action of a compact Lie group $K$. The compactness of $K$ is a crucial aspect of our proof.

When $G$ is not compact, the cohomology $\coho H p {M/G} $ is not the cohomology of the orbit space $M/G$, which may be a topologically wild space, but rather the basic cohomology of the foliation determined by the action of $G$.
\end{abstract}

\maketitle

\section{Introduction}\label{sec1}

We consider in this note a locally free smooth action $\Phi \colon G \times M \to M$ of a  connected Lie group $G$ on a manifold $M$. 
We also assume the existence of a smooth  extension of $\Phi$ to the action
$\Psi \colon K \times M \to M$ of a compact Lie group $K$.
The Cartan filtration classifies the differential forms of $M$ based on the number of vector fields tangent to the orbits of $\Phi$ required to annihilate them. This filtration defines a spectral sequence, the \emph{Leray-Serre-deRham spectral sequence}  $\hiru{E}{p,q}{r}$, which converges to the cohomology of $M$. The main result of this note is the isomorphism
\begin{equation}\label{ee}
\hiru{E}{p,q}{2} \cong \coho Hp{M/G} \otimes \coho H q \g,
\end{equation}
where $ \coho Hp{M/G} $ denotes the basic cohomology of the foliation determined by the orbits of the action $\Phi$, and $\g$ is the Lie algebra of $G$.

This result has already been proven by other authors \cite{MR658304,MR974012,razny} using different techniques such as Mayer-Vietoris on $M/G$, harmonic analysis on $M$, etc.
The compactness of $K$ is the keystone of our approach.

First of all, we prove that, thanks to the compactness of 
$K$, we can restrict ourselves to using invariant differential forms to compute $\hiru{E}{p,q}{2}$ (see \propref{61P}).
In other words, we begin with
$$
 \hiru{\BE}{p,q}{0} 
 \cong
 \left(\lau \Om p{hor} M \otimes \bi \bigwedge q(\chii_1, \ldots,\chii_n)\right)^G.
 $$
We then carry out an explicit computation of the second term of the spectral sequence by simply using the classical decomposition of the differential
$d = d_{2,-1} + d_{1,0}+d_{0,1}$ associated to a Riemannian  metric of $M$. A priori, the operator $d_{0,1}$ mixes the tangent bundle of the action with its orthogonal bundle
(see \eqref{01}). 
Once again, the compactness of $K$ allows us to construct a rich Riemannian  metric, which we call a good metric (see \propref{goodmetric}), for which $d_{0,1}$ lives in the tangent bundle of the action (see \propref{47}). 
This gives 
$$
 \hiru{\BE}{p,q}{1} 
 \cong
\left(
 \lau \Om p {hor} M\otimes \coho H q \g
  \right)^G 
=
   \coho {\Om} p {M/G} \otimes \coho H q {\g} $$
   (cf. \eqref{kkbb}).
 The computation of $\hiru{E}{p,q}{2}$  is thus achieved  and we get \eqref{ee} (see \propref{61P} and \thmref{Th}).

\section{The group $G$}\label{222}
\begin{quote}\emph{We summarize here all the tools coming from a Lie group $G$, connected subgroup of a compact Lie group $K$, that we will use in this work, including the fundamental vector fields and the associated characteristic forms.
}
\end{quote}

 Let $G$ be a connected subgroup of a compact Lie group $K$.
We consider a bi-invariant Riemannian  metric $\nu$ on $G$, which exists because $K$ is compact. The Lie algebra of $G$ is denoted by $\g$.
The associated \emph{fundamental vector field}, relative to the right action of $G$ on itself\footnote{We could have interchanged left with right with the same result.}, for a given $u\in \g$, is denoted by $Y_u$. 
These vector fields verify:
\begin{align*}
(L_g)_* Y_u =& Y_u ,
\hbox{ and}\\
 (R_g)_*Y_u =& Y_{\Ad(g^{-1})\cdot u}
\end{align*}
for any $u\in \g$ and $ g\in G$. Here 
\begin{itemize}
\item $L_g \colon G \to G$ is defined by $L_g(\ell) = g\cdot \ell$,
\item $R_g \colon G \to G$ is defined by $R_g(\ell) = \ell \cdot g^{-1}$, and
\item 
 $\Ad\colon G \to \Aut(\g)$ is the adjoint representation of $G$.
 \end{itemize}
 Notice that the vector fields $Y_\bullet$ are $G$-left invariant vector fields.

Given $u\in \g$ we write $\gamma_u =  i_{Y_u} \nu \in  \coho \Om 1 G$ the associated \emph{characteristic form} associated with $u$.
They verify 
\begin{align}\label{Rdelta}
L_g^*\gamma_u =& \gamma_u, \hbox{ and }\\ \nonumber
R_g^*\gamma_u =& \gamma_{\Ad(g^{-1})\cdot u}
\end{align}
for any $u\in \g$ and $ g\in G$.
 Notice that these differential forms $\gamma_\bullet$  are $G$-left invariant.

  Identifying the Lie algebra $\g$ with  the Lie algebra generated by the family $\{ X_1 \ldots ,X_n\}$, the complex  $ \bigwedge^* \g^*$ is in fact isomorphic to the complex of left-invariant differential forms of $G$. 
We have the isomorphism
\begin{equation}\label{gg}
 \coho H * {\g} \cong \left( \bigwedge^* \g^*\right)^G
\end{equation}
where $( \bigwedge^* \g^*)^G$  is the complex of G-invariant elements,
where we have considered the right action of G (see, for example, \cite[Chap. 4, Sec.2]{MR0336651} ). Notice that
$\left( \bigwedge^* \g^*\right)^G$ is then identified (and is isomorphic as a dgca)
 with the complex of bi-invariant diferential forms on G.

\negro
\smallskip

Consider the metric $\langle -,-\rangle$ defined on the Lie algebra $\g$ by $ \langle u,v\rangle =\nu (Y_u,Y_v)$. 
We fix for the sequel an orthonormal basis $\{u_1, \ldots, u_n\}$ for  $\g$.  For simplicity, we denote
$
Y_{u_i} = Y_i,
$
and 
 $
 \gamma_{u_i} = \gamma_i,
 $
for each $i\in \{1, \ldots n\}$. Notice that $\gamma_i (Y_j)= \delta_{i,j}$ for any $i,j \in \{1, \ldots n\}$.
 This metric $\langle -,-\rangle$ is $\Ad$-invariant since for each $g\in G$ and each $u,v\in \g$, we have:
\begin{align}\label{InvAd}
\langle \Ad(g^{-1}) \cdot u, \Ad(g^{-1})\cdot v\rangle = \nu (Y_{\Ad(g^{-1}) \cdot u}, Y_{\Ad(g^{-1})\cdot v}\rangle 
   =
 \nu ((R_{g^{-1}})_*Y_{ u}, (R_{g^{-1}})_* Y_{ v}\rangle
  = \nu( Y_{ u},Y_{ v} )  =  \langle { u},{ v} \rangle ,
\end{align}
since $\nu$ is bi-invariant.

 \begin{proposition}\label{21delta}
We have
 \begin{equation}\label{dedbis}
 d \gamma_u = -
 \sum_{1\leq a <b\leq n}  \langle [u_a,u_b],u\rangle\ \gamma_{u_a} \wedge \gamma_{u_b} 
 \end{equation}
 for each $u \in \g$.
 \end{proposition}
 \begin{proof}
We have the expression
$\displaystyle
d \gamma_u = \sum_{1\leq a <b\leq n} \alpha_{a,b} \ \gamma_{u_a} \wedge \gamma_{u_b}
$
with  $ \alpha_{a,b} \in \R$.
Let's calculate these coefficients.
\begin{align*}
 \alpha_{a,b} &=  i_{Y_{u_b}}i_{Y_{u_a}} d\gamma_u  =  i_{Y_{u_b}}L_{Y_{u_a}}  \gamma_u -  
 i_{Y_{u_b}} \underbrace{d  \underbrace{i_{Y_{u_a}} \gamma_u }_{constant}}_{=0}
 =  i_{Y_{u_b}}L_{Y_{u_a}}  \gamma_u=  \underbrace{L_{Y_{u_a}}  \underbrace{ i_{Y_{u_b}}   \gamma_u}_{constant}}_{=0} -   i_{[Y_{u_b}, Y_{u_a}]}  \gamma_u
 \\ =& -\nu([Y_{u_b},Y_{u_a}],Y_u) = \nu(Y_{[u_b,{u_a}]},Y_u)
  =  \langle [u_b,{u_a}],u\rangle = -  \langle [{u_a},u_b],u\rangle.  \qedhere
\end{align*}    
 \end{proof}

\section {Tame actions}\label{tame}
\begin{quote}\emph{
We introduce the specific class of actions  we are interested in. These actions are obtained by restricting smooth actions of compact Lie groups on manifolds. The goal of this section is to present a canonical decomposition of a differential form on the ambient manifold $M$ in terms of its characteristic forms and horizontal forms.
}

\end{quote}
\begin{definition}
A smooth action $\Phi \colon G \times M\to M$ of a connected Lie group on a manifold is called a \emph{tame action} if the following condition is satisfied:

\begin{itemize}

\item[\textsf{(T)}] There exists a smooth action $\Psi \colon K \times M\to M$, where $K\supset G$ is a compact Lie group, extending the action of $G$. 
\end{itemize}
\end{definition}

Notice that a smooth action $\Phi \colon G \times M\to M$ of a connected Lie group on a manifold $M$ is tame  if either $G$ is compact or $M$ is compact and the action is by isometries (see \cite[Theorem 1.2 of Chap. II]{MR1336823}).
The group $K$ of the definition is not unique, but it may be chosen with richer properties.

\begin{definition}
Let  $\Phi \colon G \times M\to M$ be a tame action. An action $\Psi$ satisfying  condition \textsf{(T)} is a \emph{tamer} of $\Phi$  if $G$ is dense in $K$.
\end{definition}

Any tame action
admits a tamer. Note that the Lie group $K$ is connected because $G$ is connected.
Furthermore,  $G$ is necessarily a normal subgroup of $K$ (see \cite[Proposition 1.4.1]{MR44530}).

The \emph{fundamental vector field}  (on $M$) associated to $u\in \g$ will be  denoted by $X_u$.   
These vector fields are preserved by the action of $K$ in the following way:
\begin{equation}\label{Xg}
g_*X_u = X_{\Ad(g)\cdot u} 
\end{equation}
for each $g\in K$. 
   This makes sense since $\Ad(g)\cdot u \in \g$ because $G$ is normal in $K$. \negro
For the sake of  simplicity, we write $X_i = X_{u_i}$, where $\{u_1, \ldots,u_n\}$ is the orthogonal basis of $\g$ we have chosen in the previous section.

Given a Riemannian  metric $\mu$ on $M$ we have, for any $u \in \g$, the \emph{characteristic form } $\chi_u = i_{X_u} \mu\in \coho \Om1 M$. 
We write $\chi_i = \chi_{u_i}$ for $i\in \{1, \ldots,n \}$. If the metric $\mu$ is $G$-invariant then
 \begin{equation}\label{chichi}
g^*\chii_u = \chii_{\Ad(g^{-1})\cdot u}
\end{equation}
for each $g\in G$ and $u\in \g$. 
   This makes sense since $\Ad(g)\cdot u \in \g$ because $G$ is normal in $K$. \negro
 In some cases, a careful choice of $\mu$ provides for well-behaved characteristic forms.

\begin{proposition}\label{goodmetric}
Given a locally free tame action $\Phi \colon G \times M \to M$ there exists a $G$-invariant Riemannian metric on $M$ such that
\begin{equation}\label{Xinv}
\chii_{u} (X_v) =  \langle u,v\rangle,
\end{equation}
for each $u,v\in \g$. Such a metric will be called a  \emph{ good metric}.
\end{proposition} 
\begin{proof}
 Let $\mathcal{D}$ the distribution generated by $\Phi$.
 Since the action  $\Phi$  is locally free free  on $M$, for each $x\in M$, the  family $\{ X_i(x) \tq i\in \{1, \ldots n\}\}$  forms a basis of $\mathcal{D}_x$. This distribution is $G$-invariant (cf. \eqref{Xg}).

The next step is to construct a convenient Riemannian  metric on the manifold $M$. Since the Lie group $K$ is compact, then there exists a $G$-invariant Riemannian  metric $\tau$ on $M$.
Consider  now the $\tau$-orthogonal
$G$-invariant decomposition
$
T M=  \mathcal{D} \oplus \mathcal{E}.
$
We define the $G$-invariant Riemannian  metric $\mu$ on $ M$ by 
$$
\mu (w_1, w_2) =
\left\{
\begin{array}{cl}
\tau(w_1,w_2) & \hbox{if } w_1,w_2 \in  \mathcal{E}_x \\
0 & \hbox{if } w_1 \in  \mathcal{E}_x, w_2 \in\mathcal{D}_x \\
\delta_{i,j}& \hbox{if } w_1 = X_i(x), w_2 = X_j(x).
\end{array}
\right.
$$
 This metric is $G$-invariant since $\tau$ is $G$-invariant and due to  property \eqref{InvAd}.
Notice that
$
\chii_{i} (X_j) = \mu (X_i,X_j)=  \delta_{i,j},
$
for each $i,j\in\{1,\ldots,n\}$, which gives \eqref{Xinv}.
\end{proof}

\bigskip

{\bf For the remainder of this work, we  fix  a locally free tame action $\Phi \colon G \times M \to M$ endowed with a tamer 
$\Psi \colon K \times M \to M$. We also fix an associated  good metric $\mu$ on $M$.}

\bigskip

Recall that a form $\omega \in \coho \Om * M$  is a {\em horizontal form} if $i_{X_j} \omega =0$ for each $j\in \{1,\ldots,n\}$.

\begin{proposition}\label{P1}
Each differential form $\om \in
\lau{\Om}{*}{}{ M }$ possesses a unique writing,
\begin{equation}\label{cano}
\om = \omega_{hor} + \sum_{q=1}^n \ \sum_{1\leq i_1 < \cdots < i_q\leq n } \om_{i_1, \ldots,i_q}  \wedge \chii_{i_1} \wedge \cdots \wedge \chii_{i_q},
\end{equation}
where $\omega_{hor}$ and the coefficients $  {\omega_\bullet } $ are   horizontal forms.

The {\em canonical decomposition} of $\om$ induces the isomorphism of dgcas
\begin{equation}\label{iden}
\left(\coho \Om * M,d\right)  \cong \left(\lau \Om *{hor} M \otimes \bi \bigwedge *(\chii_1, \ldots,\chii_n),d\right),
\end{equation}
where  $\lau \Om *{hor}  M$  is the graded commutative algebra of horizontal forms on $M$.
\end{proposition}

\begin{proof}
Let $q\geq 0$ be the integer verifying 
\begin{itemize}
\item $\omega (X_{i_0}, \ldots ,X_{i_q}) =0$ for any $\{i_0, \ldots, i_q\} \subset \{1 \ldots,n\}$, and
\item $\omega (X_{i_1}, \ldots ,X_{i_q}) \ne 0$ for some $\{i_1, \ldots, i_q\} \subset \{1 \ldots,n\}$.
\end{itemize}
We say that $q$ is the {\em filtration degree} of $\omega$, written $\llbracket \omega \rrbracket$.
When $\llbracket \omega \rrbracket=0$ then $\omega$ is a horizontal form and its canonical decomposition is the tautology $\omega = \omega_{hor}$.

Let us suppose $\llbracket \omega\rrbracket >0$. Consider the differential form
$$
\eta = \om -  \sum_{1\leq i_1 < \cdots < i_q\leq n } (-1)^{pq} \  i_{X_{i_q}} \cdots i_{X_{i_1}} \om  \wedge \chii_{i_1} \wedge \cdots \wedge \chii_{i_q}.
$$
Since $\llbracket \eta\rrbracket  < \llbracket \omega\rrbracket $, it suffices to apply induction  on the filtration degree to obtain \eqref{cano}. The decomposition \eqref{iden} is straightforward.
\end{proof}

\begin{proposition}\label{P2}
The canonical decomposition of  the differential of a horizontal form $\omega$ is:
 \begin{equation}\label{horib}
 d \omega = (d\omega)_{hor} + (-1)^{|\omega|}  \sum_{i =1}^n  L_{X_i} \omega \wedge\chii_i.
 \end{equation}
 \end{proposition}
  \begin{proof}
The canonical decomposition of $d\om$ is $\displaystyle d \omega = (d\omega)_{hor} + \sum_{i =1}^n (d \omega)_i  \wedge \chii_i$ (cf. \eqref{cano}). This implies
$$
 (-1)^{|\om|} (d \omega)_\ell =  i_{X_\ell} d\om = L_{X_\ell} \om,
$$
since $\om$ is a horizontal form, for each $\ell \in \{1, \ldots,n\}$. This gives the result.
\end{proof}

In this work, we use  the identification
$\eta \otimes \chii_{i_1}\wedge  \ldots \wedge  \chii_{i_q}  =\eta \wedge \chii_{i_1}\wedge  \ldots \wedge  \chii_{i_q} $ 
for any horizontal form $\eta$ and each $1\leqq i_1 < \cdots < i_q\leq n$.

We end this section by computing the differential of characteristic forms, 
a calculation that will be useful for the remainder of this work.

The complex $\bigwedge^* (\chii_{u_1}, \ldots, \chii_{u_n})$ is endowed with the differential $\delta$  defined from
  \begin{equation}\label{ded}
 \delta \chii_u =
 \sum_{1\leq a <b\leq n} \langle [u_a,u_b],u\rangle  \  \chii_{u_a} \wedge \chii_{u_b}
 \end{equation}
for each $u\in \g$  (see \eqref{dedbis}).

\begin{proposition}\label{ddelta}
For each  $i\in \{1,\ldots,n\}$,  we have a horizontal  form $e_i \in \coho  \Om 2  M$, called \emph{Euler form},  verifying
 \begin{equation}\label{bat}
d\chii_i =
e_i  + 
\delta \chii_i.
\end{equation}
\end{proposition}
\begin{proof}
We have the canonical decomposition 
$$
d\chii_{i} = e_i + \sum_{k=1}^n \alpha_k^i  \wedge \chii_{k} + \sum_{1\leq a <b\leq n} \alpha^i_{a,b} \wedge \chii_{a} \wedge \chii_{b}.
$$
where $e_i , \alpha_k^i ,  \alpha^i_{a,b} $ are horizontal forms.

The last coefficients  can be computed as follows
\begin{align*}
 \alpha^i_{a,b} =  i_{X_b}i_{X_a} d\chi_i  =  i_{X_b}L_{X_a}  \chi_i -  \underbrace{i_{X_b} d\underbrace{ i_{X_a} \chi_i }_{constant}}_{=0} 
 =  i_{X_b}L_{X_a}  \chi_i =  \underbrace{L_{X_a} \underbrace{ i_{X_b}  \chi_i}_{constant}}_{=0} -  i_{[X_b, X_a]}  \chi_i 
 =
 - i_{[X_b, X_a]}  \chi_i  = - \chi_i([X_b, X_a]) 
   =_{\eqref{Xinv}} - \langle u_i, [u_b,u_a]\rangle =  \langle[u_a,u_b],u_i\rangle.
 \end{align*} 
 
 So,
$
\displaystyle 
d\chi_i =
e_i+ \sum_{k=1}^n \alpha_k^i  \wedge\chi_{k} +
 \sum_{1\leq a <b\leq n}  \langle[u_a,u_b],u_i\rangle \wedge \chi_{a} \wedge \chi_{b}.
$
We  now prove  that the coefficients $\alpha_k^i$  are 0. To do this, we calculate $ L_{X_{k}} \chii_{i} $, with $i,k\in\{1, \ldots,n\}$,
 in two different ways.

\begin{align*}
L_{X_{k}} \chii_{i}&= L_{X_{k}} i_{X_{i}} \mu = i_{X_{i}}\underbrace{L_{X_{k}} \mu}_0   + i_{[X_{k},X_{i} ]} \mu 
 = i_{[X_{k},X_{i}]} \mu  = \chii_{[u_k,u_i]}. \\
L_{X_{k}} \chii_{i}&= i_{X_{k}} d\chii_i  =  -\alpha_k^i  +  \sum_{1\leq k <b\leq n} \langle[u_k,u_b],u_i\rangle\ \chii_{b} -   
 \sum_{1\leq a <k\leq n}  \langle[u_a,u_k],u_i\rangle\  \chii_{a} 
 =  - \alpha_k^i  +  \sum_{1\leq c\leq n} \langle[u_k,u_c],u_i\rangle  \ \chii_{c}.
\end{align*}
This gives $\alpha_k^i=0$ for each $i,k\in\{1, \ldots,n\}$.

Finally, we obtain
$$
d\chi_i = 
e_i+ \sum_{k=1}^n \alpha_k^i  \wedge \chi_{u_k} -
 \sum_{1\leq a <b\leq n}  \langle[u_a,u_b],u_i\rangle \wedge \chi_{a} \wedge \chi_{b}
 =
 e_i -
 \sum_{1\leq a <b\leq n}  \langle[u_a,u_b],u_i\rangle \wedge \chi_{a} \wedge \chi_{b}
=_{\eqref{ded}}
 e_i - \delta \chi_i. \qedhere
$$
 \end{proof}

\begin{proposition}\label{sei}
We have
\begin{equation}\label{seib}
\delta(\chii_{i_1} \wedge \cdots \wedge \chii_{i_q})
 =    (-1)^{q}  \frac {1}  2  \sum_{\ell=1} ^n L_{X_\ell}( \chii_{i_1} \wedge \cdots \wedge \chii_{i_q} ) \wedge \chii_\ell  ,
 \end{equation}
for any $1 \leq i_1 < \cdots i_q \leq n$.
\end{proposition}
\begin{proof}
The case $q=1$ follows from this sequence of equalities
\begin{align*}
\sum_{\ell=1}^n L_{X_{\ell}} \chii_{i}\wedge \chi_{\ell} &=
\sum_{\ell=1}^n i_{X_{\ell}}  d \chii_{i}\wedge \chi_{\ell}=_{\eqref{bat}}
 \sum_{\ell=1}^n i_{X_{\ell}}  \delta \chii_{i}\wedge \chi_{\ell}
=_{\eqref{ded}}
\sum_{\ell=1}^n i_{X_{\ell}}   \left(\sum_{1\leq a <b\leq n} \langle [u_a,u_b],u_i\rangle \otimes \chii_{a} \wedge \chii_{b} \right) \wedge \chi_{\ell}\\
&=
\sum_{\ell=1}^n  \sum_{1\leq \ell  <b\leq n} \langle [u_\ell ,u_b],u_i\rangle \otimes \chii_{b}  \wedge \chii_{\ell}
-
\sum_{\ell=1}^n   \sum_{1\leq a <\ell \leq n} \langle [u_a,u_\ell],u_i\rangle \otimes \chii_{a}  \wedge \chii_{\ell}
\\
&=
- \sum_{\ell=1}^n  \sum_{1\leq \ell  <b\leq n} \langle [u_\ell ,u_b],u_i\rangle \otimes \chii_{\ell}  \wedge \chii_{b}
-
\sum_{\ell=1}^n   \sum_{1\leq b <\ell \leq n} \langle [u_\ell,u_b],u_i\rangle \otimes \chii_{\ell}  \wedge \chii_{b}
=
- 2\sum_{\ell=1}^n  \sum_{1\leq \ell  <b\leq n} \langle [u_\ell ,u_b],i\rangle \otimes \chii_{\ell}  \wedge \chii_{b}
=_{\eqref{ded}}
-2\delta \chii_{i},
\end{align*}
for any $i \in \{1, \ldots,n\}$. Notice the parallelism with  \eqref{dedbis}.

For  $q\geq 1$ we have 
 \begin{eqnarray*}
 \sum_{\ell=1} ^n L_{X_\ell}( \chii_{i_1} \wedge \cdots \wedge \chii_{i_q} ) \wedge \chii_\ell  
 &=&
  \sum_{\ell=1} ^n \sum_{j=1}^n   \chii_{i_1} \wedge \cdots  \wedge L_{X_\ell} \chii_{i_j} \wedge \cdots \wedge \chii_{i_q}  \wedge \chii_\ell  = -2  \sum_{j=1} ^n (-1)^{q -j} \chii_{i_1} \wedge \cdots  \wedge  \delta \chii_{i_j}  \wedge \cdots \wedge \chii_{i_q} 
   \\& =&
 2  (-1)^{q }  \sum_{j=1} ^n (-1)^{j-1} \chii_{i_1} \wedge \cdots  \wedge    \delta \chii_{i_j}  \wedge \cdots \wedge \chii_{i_q} 
=
  2 (-1)^{q} \delta(\chii_{i_1} \wedge \cdots \wedge \chii_{i_q}). \qedhere
   \end{eqnarray*} \qedhere
\end{proof}

\section{Cartan filtration}

\begin{quote}
\emph{
We define the Cartan filtration and the induced Leray-Serre-deRham  spectral sequence, whose second term is the focus of our work.
}
\end{quote}

The \emph{Cartan filtration} of $\lau{\Om}{*}{}{M}$ associated with the  action we work with is defined by
$$
\bi{F}{p}\lau{\Om}{p+q}{}{M}  = \left\{ \om \in \lau{\Om}{p+q}{}{M} \tq
i_{X_{u_0}} \cdots i_{X_{u_q}}\om= 0
\hbox{ for each family } \{ u_0, \ldots,u_q\}  \subset \g\right\} .
$$
Notice that we have
\begin{align*}
 \bi{F}{p}\lau{\Om}{p+q}{}{M}  = &\left\{ \om \in \lau{\Om}{p+q}{}{M} \tq \llbracket \omega\rrbracket  \leq q \right\} \\ =& \left\{ \sum_{1\leq i_1 < \cdots < i_q\leq n } \om_{i_1, \ldots,i_q}  \wedge \chii_{i_1} \wedge \cdots \wedge \chii_{i_q} \tq  \hbox{  $\om_\bullet$  horizontal forms}\right\}.
\end{align*}

This filtration is a decreasing one satisfying $d\bi{F}{p}\lau{\Om}{p+q}{}{M} \subset \bi{F}{p}\lau{\Om}{p+q+1}{}{M}$. The associated spectral sequence 
$\hiru E {p,q} r $
is the \emph{Leray-Serre-deRham (LSdR) spectral sequence}. It converges to the deRham cohomology $\coho H*M$ of $M$ since $\bigcup_{p\geq 0} \bi{F}{p}\lau{\Om}{*}{}{M} =\coho \Om *M$.

The case $q=0$ is special. We get
$$
\bi{F}{p}\lau{\Om}{p}{}{M} = \{ \om \in \lau{\Om}{p}{}{M} \tq
i_{X_u} \om= 0 \hbox{ with }
u\in \g
\}
= \lau\Om p {hor}  M.
$$
  This complex is not a dgca, but it is a $K$-invariant gca:
\begin{equation}\label{hor}
i_{X_u} g^*\omega = i_{g_*X_u} \omega =_{\eqref{Xg}}  i_{X_{\Ad(g) \cdot u}\circ k}\omega =0,
\end{equation}
for any horizontal form $\omega$, $g\in K$ and $u\in \g$   (cf. \eqref{Xg}).\negro 
Using the canonical decomposition \eqref{cano} of a differential form we introduce the operator  $d_{hor} \colon \coho \Omega * M \to \lau \Omega {*+1} {hor} M $ 
defined by
\begin{equation}\label{dhor}
 d_{hor}\om = (d\om)_{hor}.
 \end{equation}

 We also denote by $\coho \Om *{M/G} = \{ \om \in \coho \Om * M \mid i_{X_u} \omega = L_{X_u} \omega =0 \hbox{ for each } u\in \g \} =   \{ \om \in \lau \Om * {hor} M \mid  L_{X_u} \omega =0 \hbox{ for each } u\in \g \}$ the differential complex of \emph{basic forms} associated to the foliation determined by the action $\Phi$. The cohomology of this complex is denoted by
$\coho H *{M/G}$  and it is called the \emph{basic cohomology}.
When $G$ is compact, this cohomology is isomorphic to the singular cohomology of the orbit space $M/G$ (cf. \cite{MR958087}).

\bigskip
Another useful decomposition in this work arises from 
$$
\coho \Om{p,q} M = \left\{  \ \sum_{1\leq i_1 < \cdots < i_q\leq n } \om_{i_1, \ldots,i_q}  \otimes \chii_{i_1} \wedge \cdots \wedge \chii_{i_q}  \in \coho \Om {p+q} M \mid \om_\bullet \hbox{ horizontal form } \right\},
$$
 and
$$
\coho \Om{p,0} M = \lau \Om p {hor} M,
$$
if $p\geq 0$.
In other words,
$
\coho \Om{p,q} M = \{ \omega \in \coho \Omega {p+q} M \tq \llbracket \omega\rrbracket  =q\}$.
This gives 
\begin{equation}\label{bideg}
\bi{F}{p}\lau{\Om}{p+q}{}{M}  = \coho \Om{p+q,0} M  \oplus \cdots \oplus \coho \Om{p,q} M .
\end{equation}
Notice that, given  $\omega \in \coho \Om{p,q} M $, and since 
$
i_{X_{i_q}} \cdots i_{X_{i_1}} \om = (-1)^{pq} \om_{i_1, \ldots,i_q} ,
$
it follows that
\begin{equation*}\label{Opq}
\om = \sum_{1\leq i_1 < \cdots < i_q\leq n } (-1)^{pq} i_{X_{i_q}} \cdots i_{X_{i_1}} \om  \wedge \chii_{i_1} \wedge \cdots \wedge \chii_{i_q}.
\end{equation*}
This result gives the decomposition
\begin{equation*}\label{iden2}
\coho  \Om{p,q} M = \lau \Om p{hor} M \otimes \bi \bigwedge q(\chii_1, \ldots,\chii_n),
\end{equation*}
and
\begin{equation*}\label{iden3}
\bi{F}{p}\lau{\Om}{p+q}{}{M}  = \bigoplus_{j=0}^q \lau \Om {p+q-j} {hor} M \otimes \bi \bigwedge j(\chii_1, \ldots,\chii_n).
\end{equation*}

\section{The differential $d$}
\begin{quote}
 \emph{The Cartan filtration induces a natural bidegree on the operator $d$ we study now. It becomes simpler when we work with invariant forms.}
 \end{quote}

The differential operator $d$ on $\coho \Omega *M$ can be decomposed into three components with bidegrees $(2,-1), (1,0)$ and $(0,1)$, relative to the Cartan filtration \eqref{bideg}:
$$
d = d_{2,-1} + d_{1,0} + d_{0,1}.
$$
We will also see that this decomposition is particularly simple when the differential form is invariant.
The three components can be described as follows.

\begin{description}

\item[$\bullet$ Bidegree $(0,1)$.] The operator 
$$
d_{0,1}  \colon
 \lau \Om p {hor} M \otimes \coho \bigwedge q {\chii_1, \ldots, \chii_n}
 \TO
 \lau \Om {p} {hor} M \otimes \coho \bigwedge {q+1} {\chii_1, \ldots, \chii_n}
 $$
  is given by
\begin{equation}\label{01}
d_{0,1}    =
 (d - d_{hor} )\otimes \id    + \id \otimes \delta,
\end{equation}
 (cf. \eqref{horib}, \eqref{dhor}, \eqref{bat}, \eqref{ded}).
\item[]

\item[$\bullet$ Bidegree $(1,0)$.] The operator 
$$
d_{1,0}   \colon
 \lau \Om p {hor} M \otimes \coho \bigwedge q {\chii_1, \ldots, \chii_n}
 \TO
 \lau \Om {p+1} {hor} M \otimes \coho \bigwedge {q} {\chii_1, \ldots, \chii_n}
 $$
is given by
\begin{equation}\label{10}
 d_{1,0}  = d_{hor}  \otimes \id.
\end{equation}

\item[]

 \item [$\bullet$ Bidegree $(2,-1)$.] The operator
 $$
d_{2,-1} \colon
 \lau \Om p {hor} M \otimes \coho \bigwedge q {\chii_1, \ldots, \chii_n}
 \TO
 \lau \Om {p+2} {hor} M \otimes \coho \bigwedge {q-1} {\chii_1, \ldots, \chii_n}
 $$ is defined from 
\begin{align*}\label{D}
d_{2,-1} (\eta \wedge  \chii_{i_1}\wedge  \ldots \wedge  \chii_{i_q} )=& 
\sum_{j=1}^q  (-1)^{p+j-1} \eta \wedge e_{i_j} \wedge  (\chii_{i_1}\wedge  \ldots \wedge \widehat{\chii_{i_j}} \wedge \cdots \wedge  \chii_{i_q}) \\
=& (-1)^p \eta \wedge (d  - \delta) ( \chii_{i_1}\wedge  \ldots \wedge  \chii_{i_q}) ,
\end{align*} 
(cf. \eqref{ddelta}).
This last operator will not be necessary for calculating the second term of the LSdR spectral sequence, but it is essential for computing the $d_2$  differential of that spectral sequence.
\end{description}
 We first need the following technical lemma.

\begin{lemma}\label{azer} 
Given an invariant differential form $\omega \in  \lau \Om p{hor} M \otimes \bi \bigwedge q (\chii_1, \ldots,\chii_n)$, we have the following formula:
$$
(d\otimes \id)  \omega
= (d_{hor} \otimes \id)\omega  -   2  (\id \otimes \delta)\omega.
$$
\end{lemma}
\begin{proof}
 Consider 
 $\displaystyle \omega  = \sum_{1\leq i_1< \cdots < i_q  \leq n}  \omega_{i_1 \ldots i_q}   \wedge \chii_{i_1} \wedge \cdots \wedge \chii_{i_q}  \in   \coho {\Om}{p,q}{M}$ an 
invariant differential form. Then
 \begin{equation}\label{bost}
 \sum_{1\leq i_1< \cdots < i_q  \leq n}  L_{X_k} \omega_{i_1 \ldots i_q}  \wedge \chii_{i_1} \wedge \cdots \wedge \chii_{i_q} +  \sum_{1\leq i_1< \cdots < i_q  \leq n}  \omega_{i_1 \ldots i_q}  \wedge L_{X_k}( \chii_{i_1} \wedge \cdots \wedge \chii_{i_q} )= 0
 \end{equation}
 for each $k\in\{1,\ldots,h\}$.
We have
 \begin{align*}
 (d\otimes \id) \om &= \sum_{1\leq i_1< \cdots < i_q  \leq n}  d \omega_{i_1 \ldots i_q}\wedge  \chii_{i_1} \wedge \cdots \wedge \chii_{i_q} 
  =_{\eqref{horib}}
  (d_{hor} \otimes \id) \ \omega +
 (-1)^{p}\sum_{1\leq i_1< \cdots < i_q  \leq n}  \sum_{\ell=1} ^n  L_{X_\ell} \omega_{i_1 \ldots i_q} \wedge \chii_\ell \wedge \chii_{i_1} \wedge \cdots \wedge \chii_{i_q} 
\\
  &=(d_{hor} \otimes \id)
   \ \omega + (-1)^{p+q} 
   \sum_{\ell=1} ^n   \left(\sum_{1\leq i_1< \cdots < i_q  \leq n}  L_{X_\ell}\omega_{i_1 \ldots i_q}  \wedge \chii_{i_1} \wedge \cdots \wedge \chii_{i_q} \right)\wedge \chii_\ell 
   \\
 &=_{\eqref{bost}} 
 (d_{hor} \otimes \id) \ \omega 
   + (-1)^{p+q+1}\sum_{\ell=1} ^n   \left(\  \sum_{1\leq i_1< \cdots < i_q  \leq n}  \omega_{i_1 \ldots i_q}  \wedge L_{X_\ell}( \chii_{i_1} \wedge \cdots \wedge \chii_{i_q} )\right)\wedge \chii_\ell  
   \\
   &=
(d_{hor} \otimes \id) \ \omega  +  (-1)^{p+q+1}  \sum_{1\leq i_1< \cdots < i_q  \leq n}  \omega_{i_1 \ldots i_q}  \wedge    \sum_{\ell=1} ^n L_{X_\ell}( \chii_{i_1} \wedge \cdots \wedge \chii_{i_q} ) \wedge \chii_\ell  
  \\
   &=_{\eqref{seib}} 
   (d_{hor} \otimes \id) \ \omega  +   2(-1)^{p+1}    \sum_{1\leq i_1< \cdots < i_q  \leq n} ^n   \omega_{i_1 \ldots i_q}
     \wedge  \delta(\chii_{i_1} \wedge \cdots \wedge \chii_{i_q}) 
     \\
   &=
(d_{hor} \otimes \id) \ \omega   -  2  \sum_{1\leq i_1< \cdots < i_q  \leq n} ^n   (\id \otimes \delta) (\omega_{i_1 \ldots i_q}
     \wedge  \chii_{i_1} \wedge \cdots \wedge \chii_{i_q} ) 
=
    (d_{hor} \otimes \id) \ \omega   -  2   (\id \otimes \delta) \omega.   \qedhere
 \end{align*}
\end{proof}

\begin{proposition}\label{47}
The differential of an invariant differential form $\omega \in  \lau \Om p{hor} M \otimes \bi \bigwedge q (\chii_1, \ldots,\chii_n)$ satisfies:
  \begin{equation}\label{471}
d \omega =   d_{2,-1}\om    \  + \ \underbrace{ (d_{hor} \otimes \id)\ \omega }_{ d_{1,0} \om }   -    \underbrace{    (\id \otimes \delta) \ \omega. }_{d_{0,1} \om}
\end{equation}
\end{proposition}

\begin{proof} It suffices to consider  the decomposition $d= d_{2,-1}+ d_{1,0}+d_{0,1}$, along with the previous Lemma and  equalities \eqref{01} and \eqref{10}.
  \end{proof}

\section{Invariant Cartan filtration}
\begin{quote}
\emph{We verify that it is possible to work only with the invariant forms to calculate the second term of the spectral sequence of LSdR.
}
\end{quote}

We begin by studying the relationship between invariance and the canonical decomposition  \eqref{cano}.

\begin{lemma}\label{inva}
Let  $\displaystyle \omega = \om_{hor} + \sum_{\ell=1}^{q} \ \sum_{1\leq i_1 < \cdots < i_\ell\leq n } \om_{i_1, \ldots,i_\ell}  \wedge \chii_{i_1} \wedge \cdots \wedge \chii_{i_\ell}   $ be the canonical decomposition of $\omega  \in \bi{F}{p}\lau{\Om}{p+q}{}{M}$.
Then,
$$
\omega \hbox{ is $G$-invariant } \Longleftrightarrow \hbox{$\omega_{hor}$ and  each } \sum_{1\leq i_1 < \cdots < i_\ell\leq n } \om_{i_1, \ldots,i_\ell}  \wedge \chii_{i_1} \wedge \cdots \wedge \chii_{i_\ell}  \hbox{ are $G$-invariant.}
$$
\end{lemma}
\begin{proof}
It suffices to prove necessity.  Given $k\in \{1,\ldots,n\}$ we have:
\begin{eqnarray*}
0&=& L_{X_k} \omega = L_{X_k} \omega_{hor} + L_{X_k} \ \sum_{\ell=1}^{q} \ \sum_{1\leq i_1 < \cdots < i_\ell\leq n } \om_{i_1, \ldots,i_\ell}  \wedge \chii_{i_1} \wedge \cdots \wedge \chii_{i_\ell} 
 =
L_{X_k} \omega_{hor} + \sum_{\ell=1}^{q}  L_{X_k} \ \sum_{1\leq i_1 < \cdots < i_\ell\leq n } \om_{i_1, \ldots,i_\ell}  \wedge \chii_{i_1} \wedge \cdots \wedge \chii_{i_\ell} 
 \\
 &=& L_{X_k} \omega_{hor} +  \sum_{\ell=1}^{q} \left(  \sum_{1\leq i_1 < \cdots < i_\ell\leq n } L_{X_k} \om_{i_1, \ldots,i_\ell}  \wedge \chii_{i_1} \wedge \cdots \wedge \chii_{i_\ell} 
 +
\sum_{1\leq i_1 < \cdots < i_\ell\leq n }\om_{i_1, \ldots,i_\ell}  \wedge  L_{X_k}  (\chii_{i_1} \wedge \cdots \wedge \chii_{i_\ell} ) \right).
\end{eqnarray*}
We know that $L_{X_k} \chii_i  = i_{X_k}d\chi_i =_{\eqref{bat}}  i_{X_k}\delta\chi_i  \in_{\eqref{ded}} \bigwedge^1(\chii_1, \ldots,\chii_n)$ for each $i\in\{1,\ldots,n\}$. Also, if $\eta \in \lau \Om * {hor} M$ then $L_{X_i} \eta$ is also a horizontal form since
$$
i_{X_\ell} L_{X_i} \eta =L_{X_i} \underbrace{i_{X_\ell} \eta}_{=0}    -i_{[X_\ell,X_i]} \eta= i_{[X_{u_i},X_{u_\ell}]} \eta = - i_{X_{[u_i,u_\ell]}} \eta =0
$$
for each $\ell \in \{1, \ldots,n\}$. 
We obtain the claim based on considerations of degree.
\end{proof}
\begin{definition}
The complex of invariant differential forms of $M$ is
$$
\coho \BOm * M = \left(\coho \Om * M\right)^G = \left(\coho \Om * M\right)^K,
$$
since $G$ is dense in $K$.
The Cartan filtration induces the following filtration in $\coho \BOm * M $:
$$
\bi{F}{p}\lau{\BOm}{p+q}{}{M}  = \bi{F}{p}\lau{\Om}{p+q}{}{M}  \cap \coho{\BOm} {p+q} M.
$$
The induced invariant LSdR spectral sequence $\hiru{\BE}{p,q}{2} $ converges to the cohomology of $\coho \BOm * M $, that is, to
$\coho H  * M $. This can be proven,
for example, using \cite[Chap. 4, Sec. 1, Theorem I]{MR0336651}, along with the fact that $K$ is a connected compact Lie group.
\end{definition}

Notice that \lemref{inva} gives
$$
\bi{F}{p}\lau{\BOm}{p+q}{}{M}  = \coho \BOm{p+q,0} M  \oplus \cdots \oplus \coho \BOm{p,q} M ,
$$
where $\lau \BOm{p,q} {} M = \left( \lau \Om{p,q} {}M\right)^G = \left( \lau \Om{p,q} {}M\right)^K$.
Each differential form $\omega \in \bi{F}{p}\lau{\BOm}{p+q}{}{M}$ has a unique decomposition
\begin{equation}\label{decomp}
\omega = \omega^{p+q,0} + \cdots + \om^{p,q},
\end{equation}
where $\om^{a,b} \in  \lau \BOm{a,b} {} M $ with $a\in \{p, \ldots , p+q\}$ and $b =p+q-a$.

\smallskip

Before studying the relationship between the two spectral sequences, we need the following key Lemma, which can essentially be found in \cite[Chap. 4, Sec.1, Theorem I]{MR0336651}.

\begin{lemma}\label{61} 
There exists a differential operator
$\rho \colon \coho \Om * M \to  \coho \BOm * M $ and a homotopy operator \negro
$ \ib{H}{M}   \colon \coho \Om * M \to \coho \Om {*-1} M$ verifying the following properties:
\begin{enumerate}[i)]
\item $\om -\rho(\om) = d \ib{H}{M}  \om+ \ib{H}{M}  d\om$, for each $\om \in \coho \Om * M$.
\item    $ \ib{H}{M}  (\bi{F}{p} \lau{\Om}{p+q}{}{M} ) \subset \bi{F}{p-1} \lau{\Om}{p+q-1}{}{M}$, with $p,q\in \N$.

\item We have $\rho \left(\lau{\Om}{p,q}{}{M}  \right) \subset   \lau{\Om}{p,q}{}{M}  $.

\item   $\rho \equiv \Ide$ on $\coho \BOm * M$, $\rho \circ (1\otimes \delta) = (1\otimes \delta) \circ \rho$ and 
$g^* \circ (1\otimes \delta) = (1\otimes \delta) \circ g^*$ for any $g \in G$.
\end{enumerate}
\end{lemma}

\begin{proof}
In this proof, for the convenience of the reader, we follow the notations from \cite[Chap. 4, Sec.1]{MR0336651}, which may differ from those used in the rest of our work. 
Note that, in particular, $\coho A*M$ denotes the complex of differential forms on $M$ (which we have denoted as $\coho\Om * M$ in this article).
Additionally, note that our action $\Psi$ is replaced by the action $T$ in this notation.
However, there is one exception: in
this proof $G$ and $K$
will remain the
same as previously
considered in this
article, while in \cite[Chap.~4, Sec.~1]{MR0336651} 
$G$ stands for a
compact Lie group,
which plays the role of our $K$.

The operator $\rho$ is defined by $
\rho(\omega) = \int_K T_a^*\omega \ da  = \int_K a^*\omega \ d a.
$
Property i) comes from \cite[Chap. 4, Sec. 1]{MR0336651}. 

	ii) For any manifold $N$, 
the operator $ \ib{H}{M} =  \ib{h}{M} T^*$  is defined as follows
	:
$$
 \ib{h}{M}  =  \ib{k}{M} - \ib{\wt{I}}{\Psi}\ib{\wt{h}}{M} \lambda^* \colon \coho A * {M\times N} \to \coho A {*-1}M.
$$
Taking $N=K$, we get the claim if we prove that, for each $\om \in \bi{F}{p} \lau{\Om}{p+q}{}{M}$ we have:
\begin{equation}\label{ii}
i_{X_{i_0}} \cdots i_{X_{i_q}} \ib{k}{M}( T^*\om)  =  \ i_{X_{i_0}} \cdots i_{X_{i_q}}  \ib{\wt{I}}{\Psi}\ib{\wt{h}}{M}( \lambda^*  T^*\om)  = 0
\end{equation}
for each $\{i_0, \ldots,i_q\} \subset \{0,\ldots,n\}$.
\smallskip

$  \rightsquigarrow$  The operator $\ib{k}{M}\colon \coho A * {M\times N} \to \coho  A {*-1} M$ uses integration 
 along
the fibers of the canonical projection
 $M \times N \mapsto M$, so
$$
i_{X_{i_0}} \cdots i_{X_{i_q}} \ib{k}{M}(\Om)=_{(a)}
\ib{\fint}{N}i_{(X_{i_0},0)} \cdots i_{(X_{i_q},0)} \Om \wedge \ib{\pi}{N}^*X = \ib{k}{M}(i_{(X_{i_0},0)} \cdots i_{(X_{i_q},0)} \Om)
$$
for each $\Om \in \coho A * {M\times N}$, where (a) is given by \cite[Chap. VII, Sec.5, Proposition X (1)]{MR0336650}.
Here, $\ib{\pi}{N}\colon M \times N \rightarrow N$ is the canonical projection.

$  \rightsquigarrow$ The operator $\ib{\wt{I}}{\Psi}\colon \coho A * {M\times U} \to \coho  A {*} M$ uses integration  along
the fibers of the canonical projection $ M\times U \mapsto M$, so
$$
i_{X_{i_0}} \cdots i_{X_{i_q}} \ib{\wt{I}}{\Psi}(\Om)=_{(a)}
\ib{\fint}{U}i_{(X_{i_0},0)} \cdots i_{(X_{i_q},0)}\Om \wedge \ib{\pi}{U}^*\Psi =  \ib{\wt{I}}{\Psi}(i_{(X_{i_0},0)} \cdots i_{(X_{i_q},0)}\Om)
$$
for each $\Om \in \coho A * {M\times U}$.
Here, $\ib{\pi}{U}\colon M \times U \rightarrow U$ is the canonical projection.

  $  \rightsquigarrow$  The operator $\ib{\wt{h}}{M}\colon \coho A * {M\times U} \to \coho  A {*-1} {M\times U} $ uses integration  along
the fibers of the canonical projection
 \footnote{In fact, $U$ is a suitable open subset of $N$.} 
 $M \times  U \times [0,1] \to M \times U$, so
\begin{align*}
i_{(X_{i_0},0)} \cdots i_{(X_{i_q},0)}\ib{\wt{h}}{M}(\Om)=_{(a)}&
\ib{\fint}{[0,1]} 
i_{(X_{i_0},0,0)} \cdots i_{(X_{i_q},0,0)}
(\id\times H)^*\Om \\
= & \ib{\fint}{[0,1]}(\id \times H)^* i_{(X_{i_0},0)} \cdots i_{(X_{i_q},0)}\Om  
= \ib{\wt{h}}{M}(i_{(X_{i_0},0)} \cdots i_{(X_{i_q},0)}\Om)
\end{align*}
for each $\Om \in \coho A* {M\times U\times [0,1]}$.

$  \rightsquigarrow$   The restriction operator $\lambda^*\colon \coho A * { M\times N} \to \coho  A {*} {M \times U}$ comes from the natural inclusion 
$M\times U \subset M \times N$.
Since $\lambda_* (X_u,0) = (X_u,0)$ then $\ib{i}{(X_u,0)}\lambda^* \Omega = \lambda^* \ib{i}{(X_u,0)}\Omega$ for each $u \in \g$ and each $\Om \in \coho A*{ M \times N}$.

\medskip

We get \eqref{ii} if we prove that $i_{(X_{i_0},0)} \cdots i_{(X_{i_q},0)} T^*\om=0$. Since $i_{X_{i_0}} \cdots i_{X_{i_q}} \om=0$ then it suffices to verify that the vector $T_* (X_u(x),0(g))$, where $u\in \g$ and $(x,g)\in M\times G$, is a tangent vector to the orbits of the action $T$.
A straightforward calculation gives
$$
T_* (X_u(x),0(g)) = X_{\Ad(g)\cdot u}(g\cdot x).
$$
The claim is proved and we get ii).

iii)  Using the canonical decomposition of \propref{P1}, it is enough to verify that $ \rho(\om )  (X_{j_0}, \ldots, X_{j_r})=0$, for $r > q$,  where $\om = \eta \wedge \chii_{i_1} \wedge \cdots \wedge \chii_{i_q}$ and $\eta$ is a horizontal form. 
By using \eqref{hor} and \eqref{Xg} we finally have
\begin{align*}
\rho(\om )  (X_{j_0}, \ldots, X_{j_r})&= \int_K a^* \om  (X_{j_0}, \ldots, X_{j_r}) da =
 \int_K    \left(a^* \eta \wedge a^* \chii_{i_1} \wedge \cdots \wedge a^* \chii_{i_q}\right) (X_{j_0}, \ldots, X_{j_r})da
 \\ &=
\pm \int_K  a^* \eta \wedge a^*  ( \chii_{i_1} \wedge \cdots \wedge \chii_{i_q})  (X_{j_0}, \ldots, X_{j_r})da =
\pm \int_K  a^* \eta \wedge ( \chii_{i_1} \wedge \cdots \wedge  \chii_{i_q})  (a_*X_{j_0}, \ldots, a_*X_{j_r})da
 = 0.  
\end{align*}
 \negro

  iv) The first property is immediate from the definition of $\rho$. Any $k\in K$ preserves the  complexes $\lau \Om * {hor}{M}$ and  $\bi \bigwedge * (\chii_1, \ldots,\chii_n)$ (cf. \eqref{hor} and \eqref{chichi}). Hence, if $\omega \in \coho \BOm {p,q} M$, then $\rho(\omega) \in \coho \BOm {p,q} M$.
Since $d\rho (\omega) = \rho (d\omega)$ we obtain $(1\otimes \delta) \rho (\omega) = \rho ((1\otimes \delta)\omega)$ for degree reasons. The last assertion follows similarly from  $g^* d = d g^*$.
\end{proof}

\begin{proposition}\label{61P}
The relationship between the two LSdR spectral sequences we have introduced is
$$
\lau{\BE}{*,*}{r}{M} \cong \lau{E}{*,*}{r}{M}. \hfill 
$$
for each $r\geq 2$.
\end{proposition}
\begin{proof}  
This follows from \lemref{61} i), ii) and iii)(cf.  \cite[Exercice 20, pag.78]{MR1793722}). 
 \end{proof}

  \section{Computation of the second term}
  \begin{quote}
  \emph{We provide the computation of the second term of the LSdR spectral sequence.}
  \end{quote}

Before beginning this computation, we need two remarks about the notation.

\medskip

A) The cohomology $\coho H * \g$ is isomorphic to the space of invariant forms $\left( \bigwedge^* \g^*\right)^G$ (cf. \eqref{gg}). Notice that the two following dgcas are isomorphic:
$$
\left(\bigwedge^* \g^*,d \right)  = \left(\bigwedge^* \left(\gamma_1, \ldots, \gamma_n\right),d \right) \ \ \hbox{ and } \ \ \ 
 \left(\bigwedge^* \left(\chi_1, \ldots, \chi_n\right), d\right)
 $$
via the assignment $\gamma_u \mapsto \chi_u$ (cf. \eqref{dedbis} and \eqref{ded}). Moreover, the Lie group $G$ acts on both dgcas preserving the aforementioned assignment (cf. \eqref{Rdelta} and \eqref{chichi}). Consequently, $\left(\bigwedge^* \left(\chi_1, \ldots, \chi_n\right)\right)^G$ is isomorphic to the cohomology $\coho H * \g$.

\medskip

B)  The first term of the invariant LSdR spectral sequence is
$$
   \hiru{\BE}{p,q}{1}   = \frac{    \hiru{\BZ}{p,q}{1}(M) }{  \hiru{\B}{p,q}{0}(M)  +   \hiru{\BZ}{p+1,q-1}{0}(M) },
 $$
with
  \begin{eqnarray}\label{18}
   \hiru{\BZ}{p,q}{1}(M)  &=&
    \nonumber  \left \{\eta \in \bi{F}{p}\lau{\BOm}{p+q}{}{M} \mid 	d \eta \in \bi{F}{p+1}\lau{\BOm}{p+q+1}{}{M}\right\} 
    =
    \bi{F}{p+1}\lau{\BOm}{p+q}{}{M} \oplus \left\{  \eta \in \lau{\BOm}{p,q}{}{M} \mid (\id \otimes \delta) \ \eta=0 \right\},
   \\ 
      \hiru{\B}{p,q}{0}(M)  &=& \left \{d\eta \in \bi{F}{p}\lau{\BOm}{p+q}{}{M} \mid 	\eta \in \bi{F}{p}\lau{\BOm}{p+q-1}{}{M}\right\} = d\bi{F}{p}\lau{\BOm}{p+q-1}{}{M}, \hbox{ and}
      \\  \nonumber
         \hiru{\BZ}{p+1,q-1}{0}(M)  &=& \left \{\eta \in \bi{F}{p+1}\lau{\BOm}{p+q}{}{M} \mid 	d \eta \in \bi{F}{p+1}\lau{\BOm}{p+q+1}{}{M}\right\} = \bi{F}{p+1}\lau{\BOm}{p+q}{}{M}.
         \end{eqnarray}

\begin{lemma}\label{not45}
Let  $\eta \in \coho {\BOm} {p,q} M$ with $(\id \otimes \delta) \eta =0.$
Then, there exists $\eta' \in \coho {\BOm} {p,q-1} M$ with  $$\eta -(\id \otimes \delta)\eta' \in \coho \Om p {M/G} \otimes \coho H q \g.$$
\end{lemma}
\begin{proof} 
Consider the cohomology class
$[\eta ] \in  \coho H q 
{ 
\lau \Om p {hor} M \otimes  \bigwedge^* \left(\chi_1, \ldots, \chi_n\right), \id \otimes \delta
}
=
\lau \Om p {hor} M\otimes \coho H q  \g
 $. 
 Notice that 
 \begin{equation}\label{kkbb}
 [\eta]\in 
 \left(
 \lau \Om p {hor} M\otimes \coho H q \g
  \right)^G 
  =
   \coho {\Om} p {M/G} \otimes \coho H q {\g} 
   = \coho \Om p {M/G} \otimes \left(\bigwedge^q \left(\chi_1, \ldots, \chi_n\right)\right)^G
 \end{equation}
  (cf. {\eqref{gg}}). This gives the decomposition
 $$
 \eta = \eta_0 + (\id \otimes \delta) \eta'',
 $$
where  $\eta_0 \in \coho \Om p {M/G} \otimes \left(\bigwedge^q \left(\chi_1, \ldots, \chi_n\right)\right)^G$  and 
$\eta'' \in  \lau  \Om {p,q-1} {} M$. Applying the operator $\rho$ of \lemref{61} iv) we get
 $$
 \eta = \rho(\eta)  = \rho(\eta_0)  + \rho (\id \otimes \delta) \eta'' =  \eta_0   + \underbrace{(\id \otimes \delta)  \rho( \eta''  )}_{\eta'}
 $$
with $\eta' \in  \coho {\BOm} {p,q-1} M$.
 \end{proof}

We have arrived at the main result of this work.

 \begin{theorem}\label{Th}
  Let $\Phi \colon G \times M \to M$ be a  locally free tame action and let $\hiru{E}{p,q}{2}$  be the second term of the associated LSdR spectral sequence. Then, the
map
    $$
  F \colon \coho H p{M/G} \otimes \coho {H} q \g \to \hiru{E}{p,q}{2}
  $$
  given by\footnote{For the calculations, we use $\coho H*{\g} =\left( \bigwedge^* (\chi_{1}, \ldots, \chi_{n})\right)^G $, as explained in  \secref{222}.}
  $$
  F ([\alpha] \otimes \beta) = \class (\alpha\otimes   \beta).
  $$
is an isomorphism, where $ \coho H p{M/G} $ denotes the basic cohomology of the foliation defined by the action, and $\coho {H} q \g$ is the cohomology of the Lie algebra $\g$ of $G$.
  \end{theorem}
  \begin{proof}
  The result comes from a commutative diagram
  $$
  \xymatrix{
  \lau \Om p{} {M/G} \otimes \coho {H} q \g \ar[r]_-\cong^-F \ar[d]^{d \otimes \id} & \hiru{E}{p,q}{1}  \ar[d]^{d_1}\\  
   \lau \Om {p+1} {} {M/G} \otimes \coho {H} q \g  \ar[r]_-\cong^-F   & \hiru{E}{p+1,q}{1}.}
 $$ 
  where $d_1$ becomes
 $
 d_1(\class(\alpha \wedge \beta)) =  \class(d \alpha \wedge \beta) 
 $
 (cf. \eqref{18} and \propref{47}).
 Notice that
 $$
 \hiru{E}{p,q}{1}   = \frac{     \left\{  \eta \in \lau{\BOm}{p,q}{}{M} \mid (\id \otimes \delta) \ \eta=0 \right\}}
 {    \left\{  (\id \otimes \delta) \eta \in \lau{\BOm}{p,q}{}{M} \mid  \ \eta \in \lau{\BOm}{p,q-1}{}{M} \right\}}.
 $$
 The operator $F$ is defined as follows:
 $$
 F(\alpha \otimes \beta) = \class (\alpha \wedge \beta).
 $$
It is clearly a well-defined monomorphism and satisfies $d_1 F = F (d \otimes \id)$. It is an epimorphism by \lemref{not45}. \end{proof}

  \bibliographystyle{abbrv}

\bibliography{BaseS3}

\end{document}